\documentclass{amsart}[12pt]
\usepackage[utf8]{inputenc}
\usepackage{amsmath,amssymb,amstext,amsthm, geometry, verbatim, tikz-cd, enumerate}
\usepackage[all]{xy}
\newtheorem{theorem}{Theorem}[section]
\newtheorem{lemma}[theorem]{Lemma}
\newtheorem{proposition}[theorem]{Proposition}
\newtheorem{claim}[theorem]{Claim}
\newtheorem{corollary}{Corollary}[theorem]
\theoremstyle{definition}
\newtheorem{definition}[theorem]{Definition}
\newtheorem{ex}{Example}[section]
\theoremstyle{remark}
\newtheorem{remark}[theorem]{Remark}
\numberwithin{equation}{section}

\newcommand \Z{\mathbb Z}

\newcommand \C{\mathbb C}

\newcommand \id{\mathrm{Id}}

\newcommand \mc{\mathcal}
\newcommand \mb{\mathbb}

\DeclareMathOperator{\Spec}{Spec}
\DeclareMathOperator{\Pic}{Pic}

\DeclareMathOperator{\rank}{rk}

\DeclareMathOperator{\Sym}{Sym}
\DeclareMathOperator{\Tot}{Tot}

\title{Spectral coverings without embeddings}
\author{Eric Boulter}
\address{Department of Mathematics and Statistics, University of Saskatchewan, SK, Canada~ S7N 5E6}
\email{eric.boulter@usask.ca}
\author{Steven Rayan}
\address{Centre for Quantum Topology and Its Applications (quanTA) and Department of Mathematics and Statistics, University of Saskatchewan, SK, Canada~ S7N 5E6}
\email{rayan@math.usask.ca}
\date{\today}
\begin{document}
\begin{abstract}
    In this article, we investigate a weakened version of the spectral correspondence for twisted Higgs bundles. Namely, we construct twisted Higgs bundles from a finite covering map and a vector bundle on that covering but without requiring that they match the eigen-data for some fixed twisted Higgs bundle. We investigate stability for twisted Higgs bundles constructed in this way, and compare our covering data to that of the traditional spectral cover.
\end{abstract}
\maketitle 

\section{Introduction and setup}

A finite, ramified covering of one variety $X$ of another $Y$ such that $X$ is embedded as a multi-section of the total space of a locally free sheaf $p:V\to Y$ sets in motion what is usually termed the ``spectral correspondence'', as described for the case $\dim(X)=\rank{V}=1.$ in \cite{Beauville-Narasimhan-Ramanan}. Here, the pushforward of a line bundle $M\to X$ under $p|_X$ produces a pair $((p|_X)_*M,(p|_X)_*\eta)$, where $\eta:M\to M\otimes (p|_X)^*V$ is the map given by tensoring with the tautological section of $p^*V\to\mbox{Tot}(V)$. This pair is usually known as a ``twisted Higgs bundle'' or ``Hitchin pair'' on $Y$. Conversely, one may start with such a pair on $Y$ and reconstruct $X\subset\mbox{Tot}(V)$ and $M$ as eigen-data. This article is motivated by a weakening of this setup where $X$ is not necessarily given from the outset as a multi-section of a sheaf $V$, and instead certain choices are made about the relationship of $X$ to $V$.
 
In the following, we assume that all varieties are projective and defined over $\C$. Many of our results will hold for arbitrary algebraically closed fields of characteristic $0$, but we make no explicit claim in this regard. Now, let $X$ and $Y$ be two smooth projective varieties with a finite surjective map $\pi:X\to Y$ of degree $r$, and fix an ample line bundle $H$ on $Y$. Choose a vector bundle $M$ on $X$ and a locally free sheaf $V$ on $Y$, as well as a section $\sigma\in H^0(X,\pi^*V).$ There is a natural map $\underline{\sigma}:M\to M\otimes \pi^*V$ sending $m$ to $m\otimes \sigma$, and the pushforward of this map by $\pi$ is a $V$-twisted Higgs bundle $(\pi_*M, \Phi:\pi_*M\to \pi_*M\otimes V)$. The annihilating polynomial $f_\Phi$ of $\Phi$ is of the form $f_\Phi(\eta)=\eta^r+s_1\eta^{r-1}+\ldots+s_{r-1}\eta+s_r$, where $s_i \in H^0(Y,\mathrm{Sym}^i(V))$ for $i\in \{1,\ldots, r\}$. This polynomial cuts out an induced spectral cover $C_\Phi \subseteq \Tot(V)$.

\section{The usual spectral correspondence}
Let $X$ be a smooth projective variety, and let $V$ be a locally free sheaf on $X$. A \emph{$V$-twisted Higgs bundle} on $X$ is a pair $(E,\Phi)$, where $E$ is a vector bundle on $X$ and $\Phi:E\to E\otimes V$ is a sheaf map satisfying $\Phi\wedge \Phi=0$. Given any $V$-twisted Higgs bundle, we can associate a subscheme $C_\Phi \subseteq \Tot(V)$ called the \emph{spectral cover of $\Phi$} by the vanishing of $f_\Phi:=\det(\Phi-\eta \otimes\id)\in H^0(\Tot(V),p_V^*(\Sym^{\rank{E}}(V)))$, where $p_V:\Tot(V)\to X$ is the bundle projection and $\eta$ is the tautological section of $p_V^*(V)$. The map $\pi_\Phi:C_\Phi\to X$ induced from the bundle projection $p_V$ is always finite locally free. If $C_\Phi$ is integral, then there is a rank-1 torsion-free sheaf $M$ on $C_\Phi$ so that $(\pi_\Phi)_*(M)=E$, and $(\pi_\Phi)_*(\eta_{C_\Phi}\otimes \id_M)=\Phi$. Conversely, starting with a torsion-free rank-1 sheaf $M$ on $C_\Phi$, the pair $((\pi_\Phi)_*M,(\pi_\Phi)_*(\eta\otimes \id_M)$ is a $V$-twisted Higgs bundle with spectral curve $C_\Phi$ whenever $(\pi_\Phi)_*M$ is locally free.

If, in the above construction, $M$ is chosen to have rank $r>1$, the resulting $V$-twisted Higgs bundle will have characteristic polynomial $f_\Phi^r$, and $\Phi$ will be annihilated by $f_\Phi$ \cite{Banerjee-Rayan}.

Motivated by this correspondence, we fix a finite map $\pi:X\to Y$ between smooth projective varieties, a vector bundle $V$ on $Y$, a vector bundle $M$ on $X$, and a section $\sigma\in H^0(X,\pi^*V)$. When there is an embedding $\iota:X\to \Tot(V)$ so that $\pi=\pi_\Phi\circ \iota$ and $\sigma=\iota^*\eta$, then this is exactly the construction from the spectral correspondence.

\section{Examples}
\begin{ex}[$\sigma \in \pi^*H^0(B,V)$]\label{pullback}
    Suppose that $\sigma \in \pi^*H^0(B,V)$, and set $U\subseteq X$ to be the complement of the ramification locus. Consider the Higgs field $\Phi|_{\pi(U)}$. For any point $p \in \pi(U)$, we can easily check that $\Phi|_p$ is a multiple of the identity, since for any $q \in \pi^{-1}(p)$ the map $\sigma$ acts by tensoring the vector $M|_q$ with $\sigma(p)$. Since $\Phi$ is equal to $\sigma\otimes \id_{\pi_*M}$ on a dense open subset, we must have $\Phi=\sigma\otimes \id_{\pi_*M}$.
\end{ex}

From now on, we assume for simplicity that $\sigma \in H^0(X,\pi^*V)\setminus\pi^*H^0(Y,V)$. To deal with sections of this form, it helps to consider the Tschirnhausen bundles associated to our covers.
\begin{definition}\label{Tschirnhausen}
    Let $\pi:X\to Y$ be a finite flat morphism between projective varieties. The \emph{Tschirnhausen bundle} of $\pi$ is the unique vector bundle $E_\pi$ on $Y$ so that $\pi_*\mc{O}_X\simeq \mc{O}_Y\oplus E_\pi^\vee$.
\end{definition}
\begin{remark}
    There is inconsistency in the literature of how the Tschirnhausen bundle is defined; some authors give it as $\pi_*\mc{O}_X/\mc{O}_Y$, and others as the dual of this bundle. We have chosen the latter convention to simplify the notation related to covers embedded in the total space of the Tschirnhausen bundle.
\end{remark}

One useful feature of the Tschirnhausen bundle is that $X$ admits a natural embedding into $\Tot(E)\simeq \Spec_{\mc{O}_Y}\left(\Sym(E^\vee)\right)$ corresponding to the ideal generated by elements of the form $v\otimes w-v\cdot w, v,w \in E(U)$, where $\cdot$ is the multiplication from the algebra structure of $\pi_*\mc{O}_X$.

\begin{ex}[Conic double cover of $\mb{P}^1$]\label{Conic}
    Suppose now that $X$ and $Y$ are both isomorphic to $\mb{P}^1$, $V$ is a line bundle, and $\pi$ is the map $[x:y]\mapsto [x^2:y^2]$. If $V\simeq \mc{O}(d)$ for $d>0$, then sections of $H^0(X,\pi^*V)/\pi^*H^0(Y,V)$ can be represented by polynomials of the form $$\sigma(x,y)=xyf(x^2,y^2),$$
    where $f$ is a homogeneous polynomials of degree $d-1$. Set $[s:t]$ to be homogeneous coordinates for $Y$. We can compute the invariant polynomials $s_1$ and $s_2$ of $\Phi$ as $\sigma(x,y)+\sigma(-x,y)=0$ and $\sigma(x,y)\sigma(-x,y)=-stf(s,t)^2$, respectively, giving an annihilating polynomial of $$\eta^2-stf(s,t)^2=\eta^2-stf(s,t)^2.$$ Notice that if $d>1$, then the induced spectral curve $C_\Phi$ is singular with singularities at the zeros of $f$, and the normalization of $C_\Phi$ is $X$.
\end{ex}

In general, the computations in Example \ref{Conic} can be extended to any case where $\deg(\pi)=2$, using the embedding of $X$ into the total space of the Tschirnhausen line bundle $\lambda\in \Pic(Y)$. (See for example \cite{Barth-Hulek-Peters-Vandeven,Friedman98}.)
\begin{proposition}[General double cover]\label{gen-double}
    Suppose that $\pi:X\to Y$ is a degree-2 map of smooth projective varieties, $M$ is a vector bundle on $X$, $V$ is a vector bundle on $Y$, and $\sigma\in H^0(X,\pi^*V)\setminus \pi^*H^0(Y,V)$ is a section of $\pi^*V$ which is not the pullback of a section of $V$. Then the Higgs field $\Phi:\pi_*M\to \pi_*M\otimes V$ constructed by pushing forward multiplication by $\sigma$ has induced spectral cover $C_\Phi$, and the normalization of $C_\Phi$ is $X$. Furthermore, the singularities of $C_\Phi$ occur precisely at the points of $Y$ where the projection of $\sigma$ to $H^0(Y, V\otimes \lambda^{-1})$ vanishes, where $\lambda$ is the Tschirnhausen bundle satisfying $\pi_*\mc{O}_X\simeq \mc{O}_Y\oplus \lambda^{-1}$.
\end{proposition}
\begin{proof}
    Suppose that $\pi$ has degree 2 and ramification divisor $R$. Then there is a line bundle $\lambda \in \Pic(Y)$ so that $\mc{O}_X(R)\simeq \pi^*\lambda$ and $\pi_*\mc{O}_X\simeq \mc{O}_Y\oplus \lambda^{-1}$. Swapping preimages of $\pi$ gives rise to an involution $\iota:X\to X$ whose fixed set is exactly the support of $R$. We have a natural decomposition of $H^0(C,\pi^*V) \simeq H^0_e(C,\pi^*V)\oplus H^0_o(C,\pi^*V)$ into sections which are fixed by $\iota^*$ and those which are negated by $\iota^*$, and this decomposition directly corresponds to the decomposition $$H^0(X,\pi^*V)=H^0(Y,\pi_*\pi^*V)\simeq H^0(Y,V\oplus V\otimes\lambda^{-1})=H^0(Y,V)\oplus H^0(Y,V\otimes\lambda^{-1})$$ by writing sections of $H^0(X,\pi^*V)$ as $\pi^*f+s\pi^*g$, where $s$ is a section of $\pi^*\lambda$ with $\mathrm{div}(s)=R$, $f$ is a section of $V$, and $g$ is a section of $V\otimes \lambda^{-1}$. For any section $\sigma =\pi^*f+s\pi^*g \in H^0(X,\pi^*V)\simeq H^0(Y,V)\oplus H^0(Y,V\otimes \lambda^{-1})$ with $g\neq 0$, the corresponding Higgs field $\Phi$ will have invariant polynomials \begin{align*}s_1(\pi(p))&=2f(\pi(p))+s(p)g(\pi(p))+s(\iota(p))g(\pi(p))=2f(\pi(p)),\\
    s_2(\pi(p))&=(f(\pi(p))+s(p)g(\pi(p)))(f(\pi(p))+s(\iota(p))g(\pi(p)))=f(\pi(p))^{2}-t(\pi(p))g(\pi(p))^{2},\end{align*} where $t\in H^0(Y,\lambda^2)$ is the section so that $\pi^*t=s^2$. From this, we can see that the annihilating polynomial of $\Phi$ is $$\eta^2-2f\eta+f^2-tg^2=(\eta-f)^2-tg^2,$$
    and the Jacobian corresponding to this annihilating polynomial is $$\begin{pmatrix}
        2(\eta-f) & 2(f-\eta)J(f)-g^2J(t)-2tgJ(g)
    \end{pmatrix}=\begin{pmatrix}
        2(\eta-f) & 2(f-\eta)J(f)-g(gJ(t)+2tJ(g))
    \end{pmatrix},$$
    so that $C_\Phi$ is smooth at any point with $\eta\neq f(p)$. Recall that $X$ embeds in $\Tot(\lambda)$ as the hypersurface $\eta^2-t$, so since $X$ is smooth the Jacobian of $t$ does not vanish at a zero of $t$. Consider a point $(p,f(p)) \in C_\Phi\subset \Tot(L)$. Then either $t$ or $g$ vanishes at $p$. Clearly, if $g(p)=0$ then the Jacobian of the annihilating polynomial also vanishes at $p$, and $C_\Phi$ is singular at $(p,f(p))$. If instead $g(p)\neq 0$, then $t(p)=0$, so the Jacobian of the annihilating polynomial becomes $\begin{pmatrix}0 & -g^2(p)J(t)_p\end{pmatrix}$, so that when $g(p)\neq0$, $(p,0)$ is a singular point if and only if $J(t)_p=0$. However, since $X$ embeds in $\Tot(\lambda)$ as the hypersurface $\eta^2-t$, any point with $t(p)=0$ and $J(t)_p=0$ would correspond to a singular point of $X$. Since $X$ is smooth, we can conclude that $(p,f(p))$ is a smooth point of $C_\Phi$ if and only if $g(p)\neq 0$. 

    Since we have a natural finite map $X\to C_\Phi$ given by sending a point $(p,\eta)\in X\subset \Tot(\lambda)$ to $(p,\eta*g(p))$ which is clearly an isomorphism away from zeros of $g$, $X$ is the normalization of $C_\Phi$.
\end{proof}

We can extend some of the techniques to a more general case where $\pi:X\to Y$ is an Abelian cover, meaning that the induced field inclusion $K(Y)\hookrightarrow K(X)$ is Galois with Abelian Galois group $\Gamma$. $K(X)$ is naturally isomorphic $K(Y)[\Gamma]$ as an algebra, so as a $\Gamma$-module $K(X)$ decomposes into $$K(X)=\bigoplus\limits_{\rho \in \hat{\Gamma}}\rho,$$
where $\hat{\Gamma}$ is the Pontryagin dual of $\Gamma$. This decomposition extends to a decomposition of $\mathcal{O}_Y$-modules:
$$\pi_*\mc{O}_X\simeq \bigoplus\limits_{\rho \in \hat{\Gamma}}L_\rho^{-1}$$ for some line bundles $L_\rho \in \Pic(Y)$ with $L_1\simeq \mc{O}_Y.$ We can also associate to each $L_\rho$ the (unique up to scaling) section $s_\rho \in H^0(X,\pi^*L_\rho)$ such that $\gamma\cdot s_\rho=\rho(\gamma)s_\rho$ for any $\gamma \in \Gamma.$ Similarly, for any vector bundle $V$ on $Y$, we have an isomorphism \begin{align*}
    \bigoplus\limits_{\rho\in \hat{\Gamma}}H^0(Y,V\otimes L_\rho^{-1}) & \simeq H^0(X,\pi^*V)\\
    h_\rho &\mapsto s_\rho\pi^*h_\rho
\end{align*}
for $h_\rho \in H^0(Y,V\otimes L_\rho^{-1}),$ which decomposes sections of a bundle $V$ into corresponding irreducible representations. The Higgs bundle $(E,\Phi)=(\pi_*M, \pi_*\underline{\sigma})$ has spectral curve given by the symmetric polynomials of $\sigma(\pi^{-1}(y))$ at a point $y \in Y$. Thus the $k^{\text{th}}$ symmetric polynomial is given by $$t_k(y)=\sum\limits_{\{\gamma_1,\ldots, \gamma_k\}\subset \Gamma}\gamma_1^*\sigma(x)\cdots\gamma_k^*\sigma(x)$$ for any $x\in \pi^{-1}(y)$, which we can compute using the representation decomposition of $\sigma$ and the algebra structure of $\pi_*\mc{O}_X$.

Before we move on to examples involving triple covers, we recall some results from \cite{Miranda85} on general triple covers. (Note an important notational difference between \cite{Miranda85} and this paper: what Miranda calls the Tschirnhausen module is the dual of the Tschirnhausen bundle as described in Definition \ref{Tschirnhausen}.)
\begin{proposition}[Proposition 4.7, \cite{Miranda85}]
    Let $X$ and $Y$ be projective varieties with $\pi:X\to Y$ be a finite flat map of degree three with associated Tschirnhausen bundle $E_\pi$ (so that $\pi_*\mc{O}_X\simeq \mc{O}_Y\oplus E_\pi^\vee$). Then the line bundle associated to the branch divisor of $\pi$ is $\det(E_\pi)^2$.
\end{proposition}
\begin{lemma}[Lemma 5.1, \cite{Miranda85}]\label{triple-singularity}
    Let $Y$ be a variety, and let $\lambda\in \Pic(Y)$ be a line bundle. Suppose $p$ is a point of $Y$ with local ring $(\mc{O}_p,\mathfrak{m})$ such that $Y$ is smooth at $p$. Take $X$ to be a triple cover of $Y$ of the form $V(\eta^3+a\eta+b)\subset \Tot(\lambda)$, where $\eta$ is the tautological section of $\lambda$, $a\in H^0(Y,\lambda^2)$, and $b \in H^0(Y,\lambda^3)$. Then $X$ has a singularity above $p$ if and only if one of the following conditions holds:
    \begin{enumerate}
        \item[(a)] $a \in \mathfrak{m}$ and $b \in \mathfrak{m}^2$, in which case $(p,0)$ is a singular point of $X$;
        \item[(b)] $a \not\in \mathfrak{m}$ and $\Delta\in \mathfrak{m}^2$, in which case $(p,-3b(p)/2a(p))$ is a singular point of $X$.
    \end{enumerate}
    Here $\Delta:=4a^3+27b^2$ is the cubic discriminant of the polynomial defining $X$.
\end{lemma}
\begin{ex}[Triple cover $\mb{P}^1\to \mb{P}^1$]\label{triple-d1}
Suppose that $X$ and $Y$ are both isomorphic to $\mb{P}^1$, and $\pi$ is the map $[x:y]\mapsto [x^3:y^3]$. If $V\simeq \mc{O}(d)$ for $d>0$, then sections of $H^0(X,\pi^*V)/\pi^*H^0(Y,V)$ can be represented by polynomials of the form $$\sigma(x,y)=xy^2f(x^3,y^3)+x^2yg(x^3,y^3),$$
where $f$ and $g$ are homogeneous polynomials of degree $d-1$. Set $[s:t]$ to be homogeneous coordinates for $B$. We can compute the invariant polynomials $s_1$, $s_2$, and $s_3$ of $\Phi$ as $s_1=\sigma(x,y)+\sigma(\omega x,y)+\sigma(\omega^2x,y)=0$, $s_2=\sigma(x,y)\sigma(\omega x,y)+\sigma(x,y)\sigma(\omega^2x,y)+\sigma(\omega x,y)\sigma(\omega^2x,y)=-3stf(t,s)g(t,s)$, and $s_3=\sigma(x,y)\sigma(\omega x,y)\sigma(\omega^2x,y)=st^2f(s,t)^3+s^2tg(s,t)^3,$ giving an annihilating polynomial $$\eta^3-3stf(s,t)g(s,t)\eta-st^2f(s,t)^3-s^2tg(s,t)^3.$$ The corresponding spectral curve $C_\Phi$ will be integral with arithmetic genus $3d-2$ by Grothendieck--Riemann--Roch if $f$ and $g$ are not both identically zero. The cubic discriminant of the annihilating polynomial is given by $$\Delta(s,t)=27s^2t^2(tf(s,t)^3-sg(s,t)^3)^2.$$ We claim that the singular locus lies over the divisor $tf(s,t)^3-sg(s,t)^3$. This divisor appears with multiplicity 2 in the discriminant, so by Lemma \ref{triple-singularity} there is a singularity above any point where $tf(s,t)^3-sg(s,t)^3$ vanishes and $stf(s,t)g(s,t)$ does not. Suppose that $p$ is a common zero of $stf(s,t)g(s,t)$ and $tf(s,t)^3-sg(s,t)^3$. By Lemma \ref{triple-singularity} it is enough to show that $st(tf(s,t)^3+g(s,t)^3)$ vanishes to order 2 at any such point. 

If $tf(s,t)^3-sg(s,t)^3$ vanishes at $[0:1]$, then $f(0,1)=0$ and if $tf(s,t)^3-sg(s,t)^3$ vanishes at $[1:0]$ then $g(1,0)=0$, so $X$ has a singularity above common zeros of $tf(s,t)^3-sg(s,t)^3$ and $st$.

If $[s:t]$ is a shared zero of $f$ and $tf^3-sg^3$, then either $[s:t]=[0:1]$ or $g(s,t)=0$. Similarly, should $[s:t]$ be a common zero of $g$ and $tf^3-sg^3$, then either $[s:t]=[1:0]$ or $f(s,t)=0$. In any of these cases, $([s:t],0)$ is a singular point of $X$.

\end{ex}
The computations of the cubing map on $\mathbb{P}^1$ can be extended to work in the general situation of a cyclic triple cover.

\begin{ex}[General cyclic triple cover]
    Let $\pi:X\to Y$ be any cyclic triple cover map, and suppose that $V$ is a line bundle. Then, the Tschirnhausen bundle of $\pi$ is of the form $L_1\oplus L_2$, where $L_1^{-1}$ and $L_2^{-1}$ correspond to the $\zeta$ and $\zeta^2$ sub-representations of $\Z/3\Z$ on $\mc{O}_X$, respectively, for $\zeta$ a primitive $3^\text{rd}$ root of unity. It follows that there is a unique section $s \in H^0(X,\pi^*L_1)$ up to scaling so that $\zeta\cdot s=\zeta s$ and a unique section $t \in H^0(X,\pi^*L_2)$ up to scaling so that $\zeta\cdot t=\zeta^2 t$. Now for any vector bundle $E$ on $Y$, we can use the decomposition of the group $H^0(X,\pi^*V)$ into subrepresentations to uniquely write any section $\sigma \in H^0(X,\pi^*E)$ as $\sigma=\pi^*f+s\pi^*g+t\pi^*h$, where $f \in H^0(Y,E)$, $g \in H^0(Y,E\otimes L_1^{-1})$, and $h \in H^0(Y,E\otimes L_2^{-1})$. 

    In particular, since $\zeta\cdot s^2=\zeta^2s^2$, there is a section $a \in H^0(Y,L_1^2\otimes L_2^{-1})$ so that $s^2=t\pi^*a$. Similarly, since $\zeta\cdot t^2=\zeta t^2$, there is a section $b \in H^0(Y,L_2^2\otimes L_1^{-1})$ so that $t^2=s\pi^*b$. From this we can see that since $s^2t^2=st\pi^*(ab)$, we must have $st=\pi^*(ab)$.

    Now take $\sigma=s\pi^*g+t\pi^*h\in H^0(X,\pi^*V)$ with $g$ and $h$ not both zero. Pushing forward multiplication by $\sigma$ gives the annihilating polynomial $$\eta^3-3abgh\eta-a^2bg^3-ab^2h^3$$ through analogous computations to Example \ref{triple-d1}, and the cubic discriminant will be $$\Delta=27(a^2bg^3-ab^2h^3)^2=27a^2b^2(ag^3-bh^3)^2.$$
    Note also that the branch locus of $\pi$ is given by $a^2b^2$ by \cite[Proposition 7.4]{Miranda85}, implying that we should expect the singular locus of the induced spectral curve $C_\Phi$ to be $ag^3-bh^3$. Indeed, using the singularity analysis from Lemma \ref{triple-singularity}, we see that the induced spectral curve is singular at any point where $ag^3-bh^3$ vanishes and $abgh$ does not. For any point $p$ where $abgh$ vanishes, $a^2bg^3+ab^2h^3$ will vanish to order 2 at $p$ if and only if $ag^3-bh^3$ vanishes at $p$, so $C_\Phi$ is singular at every point of the divisor $ag^3-bh^3$. 
\end{ex}

\begin{claim}
    If $r$ is prime, the induced spectral cover $C_\Phi$ is birational to $X$ whenever $\sigma \in H^0(X,\pi^*V)\setminus \pi^*H^0(Y,V).$
\end{claim}
\begin{proof}
    We can construct a map $\psi:X \to C_\Phi\subseteq \Tot(L)$ as $\psi(x):=(\pi(x), \sigma(x))$. (Since $\sigma$ is a section of $\pi^*L$, $\sigma(x)$ naturally belongs to the fibre of $\pi(x)$ in $\Tot(V)$.) Furthermore, composing $\Psi$ with the natural projection of the spectral cover $C_\Phi$ is exactly $\pi$. Since $\pi$ is a finite map of prime degree, then $\psi$ has degree $r$ or $1$. If $\sigma$ is not a pullback section, then for a general choice of $y \in Y$ we will have $\sigma(\pi^{-1}(y))$ containing $r$ distinct points. This shows that we must have $\psi$ of degree 1, so that it is a birational morphism.
\end{proof}

The next example shows that we cannot expect the above claim to hold when $r$ is composite.
\begin{ex}
    Again, suppose that $X$ and $Y$ are both $\mb{P}^1$ and $\pi$ is the map $[x:y]\mapsto [x^4:y^4]$. If we take $V=\mc{O}(d)$ with $d\geq 1$, then any polynomial in $H^0(X,\pi^*V)/\pi^*H^0(Y,V)$ can be written as $\sigma=x^3yf(x^4,y^4)+x^2y^2g(x^4,y^4)+xy^3h(x^4,y^4)$ using the fact that $\pi$ is Galois with group $\Z/4\Z$. Notice that $\pi$ decomposes as $\pi=p\circ p$, where $p:\mb{P}^1\to \mb{P}^1$ is the map $[x:y]\mapsto [x^2:y^2]$ corresponding to the intermediate Galois cover of order 2. If $f$ and $h$ are zero in $\sigma$, then $\sigma$ belongs to $p^*H^0(p(X),p^*V)$. Let $\Psi$ be the Higgs field given by pushing forward multiplication by $\sigma$ along $p$. Since $\sigma \in p^*H^0(p(X),p^*V)$, Example \ref{pullback} tells us that $\Psi=\tau\otimes \id_{p_*M}$, where $\tau \in H^0(p(X,p^*V)$ is the unique section with $p^*\tau=\sigma$. We are now in precisely the situation of Proposition \ref{gen-double} with the map $p:\mb{P}^1\to \mb{P}^1$, the section $\tau\in H^0(P(X),p^*V)/p^*H^0(Y,V)$, and a vector bundle $p_*M$ on $P(X)$, so that $\Phi$ has annihilating polynomial of degree two and the natural map $X\to C_\Phi$ is the composition of the normalization of $C_\Phi$ with $p$.
\end{ex}
\section{Result}
\begin{theorem}\label{Normalization}
    Let $\pi:X\to Z$ be a finite surjective map of smooth projective varieties, $M$ a torsion-free sheaf on $X$, $V$ a locally free sheaf on $Z$, and $\sigma$ a section of $H^0(X,\pi^*V)$. Let $\Phi:\pi_*M\to \pi_*M\otimes L$ be the Higgs field which is the pushforward by $\pi$ of the map $\underline{\sigma}:M\to M\otimes \pi^*V$ given by $m\mapsto m\otimes \sigma$. Then there is a normal variety $Y$ and two finite maps $p:X\to Y$ and $q:Y\to Z$ so that $\sigma=p^*\tau$ for some section $\tau\in H^0(Y,q^*V)$, $\Phi$ can be constructed as the pushforward by $q$ of the morphism $\underline{\tau}:p_*M\to p_*M\otimes q^*V$ given by $m\mapsto m\otimes \tau$, and the spectral cover $C_\Phi$ has normalization isomorphic to $Y$. 
\end{theorem}
\begin{remark}
    In the above theorem, when $p$ is an isomorphism the spectral curve is birational to $X$, and when $q$ is an isomorphism, the resulting Higgs bundle has the form $\tau\otimes \id_{\pi_*M}$ for some $\tau \in H^0(Z,V)$. Of course, when both $p$ and $q$ are identity maps, then $\pi$ is an isomorphism
\end{remark}
\begin{proof}
    First, note that there are natural maps $\rho:X\to C_\Phi\subseteq \Tot(V)$ given by $\rho(x)=(\pi(x), \sigma(x))\in \Tot(V)$, which is well-defined since $\sigma$ is a section of $\pi^*V$ as well as $\pi_\Phi:C_\Phi\to Z$ defined as the restriction of the natural projection $\Tot(V)\to Z$ to $C_\Phi$. Both maps are finite, and satisfy $\pi_\Phi\circ \rho=\pi$. Let $\nu:Y\to C_\Phi$ be the normalization of $C_\Phi$. Then we have a map $p:X\to Y$ so that $\rho=\nu \circ p$ and a map $q:=\pi_\Phi\circ \nu:Y\to Z$ so that $\pi=q\circ p$. Notice that the section $\sigma$ is constant on fibres of $\rho$, and is therefore also constant on fibres of $p$, so we can construct a section $\tau \in H^0(Y,q^*V)$ so that $p^*\tau=\sigma$. Pushing forward $\underline{\sigma}$ by $p$ gives a map $\underline{\tau}:p_*M\to p_*M\otimes q^*V$ which is again multiplication by the section $\tau$, as described in Example \ref{pullback}. Clearly, pushing forward $\underline{\tau}$ by $q$ gives the Higgs field $\Phi$.
\end{proof}
\begin{remark}
    By considering the case where $q$ is the identity in the above theorem, one can check that the only nilpotent $V$-twisted Higgs bundles resulting from this construction are those with zero Higgs field.
\end{remark}
\begin{corollary}
    Let $\pi:X\to Y$ be a branched covering of prime degree, and let $\Phi$ be the Higgs field induced by the data $(\pi, M, V, \sigma)$. Then either $\Phi=\tau\otimes \id_{\pi_*M}$ for some $\tau \in H^0(Y, V)$ or $X$ is the normalization of $C_\Phi$.
\end{corollary}

\section{Stability}
Let $X$ and $Y$ be projective varieties and let $f:X\to Y$ be a finite map. Choose an ample line bundle $H$ on $Y$. In this case, $f^*H$ is also an ample line bundle, and $H^i(X,\mc{E}\otimes \pi^*H^k)\simeq H^i(Y,\pi_*\mc{E}\otimes H^k)$ for any torsion-free coherent sheaf $\mc{E}$ on $X$ and any integers $i$ and $k$. In particular, if we define the normalized Hilbert polynomial as$$p_{H,\mc{F}}(k):=\sum\limits_{i=0}^\infty (-1)^i\frac{h^i(Y,\mc{F}\otimes H^k)}{\rank(\mc{F})}$$
for a choice of ample line bundle and torsion-free sheaf, then we get the relation \begin{align}\label{finitePolynomial}p_{\pi^*H, \mc{E}}(k)=\deg(f)p_{H,\pi_*\mc{E}}(k)\end{align} for every integer $k$.

If we define the inequality $\prec (\mathrm{resp.} \preceq)$ on Hilbert polynomials as$$p_{H,\mc{F}}\prec(\mathrm{resp.} \preceq) p_{H,\mc{E}} \iff p_{H,\mc{F}}< (\mathrm{resp.} \leq) p_{H,\mc{E}} \forall k>>0,$$
then it becomes clear that for any pair of torsion-free sheaves $\mc{E}, \mc{F}$ on X, $p_{\pi^*H,\mc{F}}\prec p_{\pi^*H,\mc{E}}$ if and only if $p_{H, \pi_*\mc{F}}\prec p_{H,\pi_*\mc{E}}$.

Recall the definition of Giesecker stability: given an ample line bundle $H$, A torsion-free sheaf $\mc{E}$ is \emph{Giesecker $H$-(semi-)stable} for any subsheaf $\mc{F}\subset \mc{E}$ with $0<\rank(\mc{F})<\rank(\mc{E}),$ we have $p_{H,\mc{F}}\prec (\preceq) p_{H,\mc{F}}$.

\begin{proposition}\label{stability}
    Consider the context of Theorem \ref{Normalization}, and take $H \in NS(Z)$ to be an ample class. The $V$-twisted Higgs bundle $(\pi_*M,\Phi)$ defined by the data $(\pi,M,V,\sigma)$ is $H$-(semi-)stable if and only if the sheaf $\rho_*M$ is $\pi_\Phi^*H$-(semi-)stable.
\end{proposition}
\begin{proof}
    Recall that any finite map $f:V\to W$ of varieties induces an equivalence of categories between coherent $\mc{O}_W$ modules and coherent $f_*\mc{O}_V$ modules, and it is clear from the equation \eqref{finitePolynomial} this equivalence of categories preserves Giesecker stability for appropriately ample line bundles. 

    In the context of Theorem \ref{Normalization}, the $V$-twisted Higgs bundle $(\pi_*M, \Phi)$ is naturally a $(f_\Phi)_*\mc{O}_{C_\Phi}$ module, so it is $H$-stable if and only if $\rho_*M$ is $f_\Phi^*H$-stable.
\end{proof}

In the case that $\pi:X\to Z$ has degree $2$, we can more directly relate $H$-stability of the $V$-twisted Higgs bundle to the starting data.

\begin{proposition}\label{doublecoverstability}
    Let $\pi:X\to Z$ be a finite surjective map of degree 2 between smooth projective varieties, and fix an ample class $H \in NS(Z)$. If we take $M$ to be a vector bundle on $X$ and $V$ a vector bundle on $Z$ and choose a section $\sigma \in H^0(X,\pi^*V)\setminus \pi^*H^0(Z,V)$, then the $V$-twisted Higgs bundle defined by $(\pi,M,V,\sigma)$ is $H$-(semi-)stable if and only if $M$ is $\pi^*H$-(semi)-stable.
\end{proposition}
\begin{proof}
    Since $\pi$ has degree 2, there is a unique line bundle $\lambda \in \Pic(Z)$ so that $\pi_*\mc{O}_X\simeq \mc{O}_Z\oplus \lambda^{-1}$. There is also a natural involution $\iota:X\to X$ induced by $\pi$ which swaps sheets of the double cover. Let $s \in H^0(X,\pi^*\lambda)$ be a non-zero section so that $\iota^*s=-s$. Using the decomposition $H^0(X,\pi^*V)\simeq H^0(Z,V)\oplus H^0(Z,V\otimes \lambda^{-1})$, we can uniquely decompose $\sigma$ as $\sigma=\pi^*f+s\otimes\pi^*g$, where $f \in H^0(Z,V)$, $g\in H^0(Z,V\otimes \lambda^{-1})$, and $g\neq0$. Let $\Psi$ be the $\lambda$-twisted Higgs field corresponding to the data $(\pi, M, \lambda, s)$. It is easy to check that the spectral cover of $\Psi$ is exactly the embedding of $X$ into $\Tot(\lambda)$ induced by $s$. Using the above decomposition of $\sigma$, we can write the $V$-twisted Higgs field $\Phi$ induced by $(\pi,M,V,\sigma)$ as $\Phi=f\otimes \id_{\pi_*M}+g\otimes \Psi$. Since $\Phi$ and $\Psi$ have the same invariant subsheaves, $\Psi$ is $H$-(semi-)stable if and only if $\Phi$ is. Furthermore, applying Proposition \ref{stability} to $\Psi$ gives that $\Psi$ is $H$-(semi-)stable if and only if $M$ is $\pi^*H$-(semi-)stable, and so combining the two equivalences gives the desired result.
\end{proof}

\subsection{Pushforward from the normalization}

Let $X$ be a projective variety with smooth normalization $\nu:\tilde{X}\to X$. We set $\tilde{\mc{O}}_X:=\nu_*\mc{O}_{\tilde{X}}$. Let $\mc{J}$ be the conductor ideal sheaf of $\nu$. (As a reminder, $\mc{J}$ is defined as being maximal among ideal sheaves of $\mc{O}_X$ which are also ideal sheaves of $\tilde{\mc{O}}_X$.)
\begin{proposition}\label{Otilde-algebras}
    Let $\mc{F}$ be a torsion-free coherent sheaf on $X$ which is also an $\tilde{\mc{O}}_X$-module. Then there is a unique coherent sheaf $\mc{E}$ on $\tilde{X}$ so that $\nu_*\mc{E}\simeq \mc{F}$.
\end{proposition}
This is a generalization of \cite[Chapter 8, Proposition 10]{Seshadri}.
\begin{proof}
    Let $U$ be any affine open subset of $X$. We first show that any isomorphism $\psi:\mc{F}_1\to \mc{F}_2$ between torsion-free $\mc{O}_X(U)$-modules is an $\tilde{\mc{O}}_X(U)$-module isomorphism. For any $f \in \mc{F}_1, a\in \tilde{\mc{O}}_X(U), m \in \mc{J}(U)$, consider $\psi(m\cdot a\cdot f)$. By definition of the conductor, $m\cdot a \in \mc{O}_X(U)$, so $m\cdot\psi(a\cdot f)=\psi(m\cdot a \cdot f)=m\cdot a \cdot \psi(f)$. Since $\tilde{\mc{O}}_X(U)$ is an integral domain and $\mc{F}_2$ is torsion-free, this implies that $\psi(a\cdot f)=a\cdot \psi(f)$, implying that $\psi$ is an isomorphism of $\tilde{\mc{O}}_X(U)$-modules. This shows that any $\mc{E}$ as in the statement will be unique up to unique isomorphism. We choose the candidate $\mc{E}\simeq \nu^*\mc{F}/\mathrm{Tors}(\nu^*\mc{F})$. For any affine open subset $U$ of $X$, $\mc{E}(\nu^{-1}(U))$ is given by $\mc{F}(U)\otimes_{\mc{O}_X(U)} \tilde{\mc{O}}_X(U)/\mathrm{Tors}$. For any $f \in \mc{F}(U)$ and $a\in \tilde{\mc{O}}_X(U)$, $f\otimes a=(a\cdot f)\otimes 1+[f\otimes a -(a\cdot f)\otimes 1]$, and for any $m \in J(U)$, $$m\cdot (f\otimes a)=(m\cdot a\cdot f)\otimes 1+[f\otimes (m\cdot a) -(m\cdot a\cdot f)\otimes 1]=(m\cdot a\cdot f)\otimes 1$$
    since $m\cdot a \in \mc{O}_X(U)$. This clearly shows that the torsion-free part of $\nu^*\mc{F}(\nu^{-1}(U))$ is exactly $\mc{F}(U)$, so that $\nu_*\mc{E}\simeq \mc{F}$.
\end{proof}
\begin{corollary}\label{stability-normalization}
    Consider the context of Proposition \ref{Otilde-algebras}, and suppose that $H$ is an ample line bundle on $X$. If $\mc{F}$ is a torsion-free coherent sheaf on $\tilde{X}$ and $\nu_*\mc{F}$ is Giesecker $H$-(semi-)stable, then $\mc{F}$ is Giesecker $\nu^*H$-(semi-)stable.
\end{corollary}
The converse of this corollary holds when $\tilde{X}\simeq \mb{P}^1$ or when $X$ is a curve whose singular points are all ordinary double points \cite{Avritzer-Lange-Ribeiro,Avritzer-Ribeiro-Martins}.

\section{Questions}

We close this short article with a few research questions of interest that arise from the constructions above:

\begin{enumerate}    
    \item Can we place any restrictions on the type of singularities that $Y$ can have in the case of Theorem \ref{Normalization}? 
    
    \item Is there a ``nice'' characterization of the singularities that can occur in the spectral curves of Higgs bundles fitting the hypotheses of Proposition \ref{doublecoverstability}?

    \item Can we find counterexamples to the converse of Corollary \ref{stability-normalization}? 
    
\end{enumerate}

\bibliographystyle{alpha}
\bibliography{references}

\end{document}